\documentclass[12pt,a4paper]{article}
\usepackage[utf8]{inputenc}
\usepackage{amsmath, amsthm}
\usepackage{amsfonts}
\usepackage{amssymb}
\usepackage{makeidx}
\usepackage{graphicx}
\usepackage{hyperref}
\usepackage[width=164.00mm, height=186.00mm]{geometry}
\newcommand{\Mod}[1]{ \left(\mathrm{mod}\ #1\right)}
\newtheorem{theorem}{Theorem}[section]

\newtheorem{remark}{Remark}[section]
\newtheorem{corollary}{Corollary}[section]

\newcommand{\n}{\noindent}
\author{Ratan Lal and Vipul Kakkar \vspace{.2cm}\\  Department of Mathematics, Central University of Rajasthan,\\ Ajmer, India-305817.\\ \textit{Email:} \href{vermarattan789@gmail.com}{vermarattan789@gmail.com}, \href{vplkakkar@gmail.com}{vplkakkar@gmail.com} }
\title{Central Automorphisms of Zappa-Sz\'{e}p Products}

\begin{document}
	\maketitle
	\begin{abstract}
		In this paper, the central automorphism group of the Zappa-Sz\'{e}p product of two groups is obtained. 
	\end{abstract}



	\section{Introduction}
		Let $G$ be a group and $Aut(G)$ be the group of automorphisms of $G$. An automorphism $\theta\in Aut(G)$ is called a central automorphism of $G$ if it commutes with every inner automorphism of group $G$ or equivalently $g^{-1}\theta(g)\in Z(G)$ for all $g\in G$, where $Z(G)$ denotes the center of the group $G$. Central automorphisms are important to investigate the whole automorphism group $Aut(G)$. In fact, the set $Aut_{c}(G)$ of all central automorphisms of $G$ is a normal subgroup of $Aut(G)$. It is obvious that if we have an abelian automorphism group, then it coincides with the central automorphism group. The automorphism group of a $p$-group (where $p$ is a prime) coinciding with the central automorphism group are studied in \cite{cur82, gl86} and \cite{mal}. The study of central automorphisms of a group has been an interest to the algebraists (see \cite{cur82,mc01, gl86,jaf1, jaf2, jam, mal}).

	Let $H$ and $K$ be two subgroups of a group $G$. Then, $G$ is called the internal Zappa-Sz\'{e}p product of $H$ and $K$ if $G=HK$ and $H\cap K = \{1\}$. The Zappa-Sz\'{e}p product is a natural generalization of the semidirect product of two groups. If $G$ is the internal Zappa-Sz\'{e}p product of $H$ and $K$, then $K$ appears as a right transversal to $H$ in $G$. Let $h\in H$ and $k\in K$. Then $kh=\sigma(k,h)\tau(k,h)$, where $\sigma(k,h)\in H$ and $\tau(k,h)\in K$. This determines the maps $\sigma: K \times H \rightarrow H$ and $\tau: K\times H \rightarrow K$ defined by $\sigma(k,h) = \sigma_{k}(h)$ and $\tau(k,h) = \tau_{h}(k)$, for all $h\in H$ and $k\in K$ respectively. These maps are called the matched pair of groups and satisfy the following conditions (see \cite{af})
	
	\begin{itemize}
		\item[$(C1)$] $\sigma_{1}(h) = h$ and $\tau_{1}(k) = k$,
		\item[($C2$)] $\sigma_{k}(1) = 1 = \tau_{h}(1)$,
		\item[$(C3)$] $\sigma_{kk^{\prime}}(h) = \sigma_{k}(\sigma_{k^{\prime}}(h))$,
		\item[($C4$)] $\tau_{h}(kk^{\prime}) = \tau_{\sigma_{k^{\prime}}(h)}(k)\tau_{h}(k^{\prime})$,
		\item[($C5$)] $\sigma_{k}(hh^{\prime}) = \sigma_{k}(h)\sigma_{\tau_{h}(k)}(h^{\prime})$,
		\item[$(C6)$] $\tau_{hh^{\prime}}(k) = \tau_{h^{\prime}}(\tau_{h}(k)$,
	\end{itemize} 
	for all $h,h^{\prime} \in H$ and $k,k^{\prime}\in K$. 
	
	 Now, let $H$ and $K$ be two groups, $\sigma: K \times H \rightarrow H$ and $\tau: K \times H \rightarrow K$ be two maps which satisfy the above conditions. Then, the set $H\times K$ with the binary operation defined by 
	\begin{equation*}
	(h,k)(h^{\prime},k^{\prime}) = (h\sigma_{k}(h^{\prime}),\tau_{h^{\prime}}(k)k^{\prime})
	\end{equation*}
	forms a group called the external Zappa-Sz\'{e}p product of $H$ and $K$. The internal Zappa-Sz\'{e}p product is isomorphic to the external Zappa-Sz\'{e}p product (see \cite[Proposition 2.4, p. 4]{af}). We will identify the external Zappa-Sz\'{e}p product with the internal Zappa-Sz\'{e}p product. 
	
	G. Zappa \cite{gz} introduced the Zappa-Sz\'{e}p product of two groups which was also studied by J. Sz\'{e}p in the series of papers (few of them are \cite{sz2,sz3,sz1,sz4}). Finite simple groups are studied as a Zappa-Sz\'{e}p product of two groups with coprime orders \cite{za}. Also, it was observed that a finite group is solvable if and only if it is a Zappa-Sz\'{e}p product of a Sylow $p$-subgroup and a Sylow $p$-complement \cite{ph}.
	
	H. Mousavi and A. Shomali\cite{censd} have studied the central automorphisms of the semidirect product of groups. Extending this work, in this paper, we have found the structure of the central automorphism group of the Zappa-Sz\'{e}p product of two groups. As an application, we have computed the central automorphism group of a group of order $p^{5}$ which is the Zappa-Sz\'{e}p product of two abelian groups of orders $p^{2}$ and $p^{3}$, where $p$ is an odd prime. 
	
	Throughout the paper, $\mathbb{Z}_{n}$ denotes the cyclic group of order $n$. Let $U$ and $V$ be groups. Then $Hom(U,V)$ denotes the group of all group homomorphisms from $U$ to $V$. The group of all the surjective group homomorphisms of a group $G$ is denoted by $Epi(G)$ and $i_{g}$ defined by $x\mapsto gxg^{-1}$, for all $g\in G$ denotes the inner automorphism of a group $G$. $M_{r\times s}(\mathbb{Z}_{n})$ denotes the set of all matrices of order $r\times s$ with entries in $\mathbb{Z}_{n}$.  
	
	\section{Structure of the central automorphism group}
	\n  Let $G$ be the Zappa-Sz\'{e}p product of two groups $H$ and $K$. Let $U, V$ and $W$ be any groups. Then $Map(U, V)$ denotes the set of all maps between the groups $U$ and $V$. If $\phi, \psi \in Map(U, V)$ and $\eta \in Map(V, W)$, then $\phi + \psi \in Map (U,V)$ is defined by $(\phi + \psi)(u) = \phi(u)\psi(u)$, $\eta\phi \in Map(U,W)$ is defined by $\eta\phi(u) = \eta(\phi(u))$, $\sigma_{\phi}(\psi) \in Map(U,V)$ is defined by $(\sigma_{\phi}(\psi))(u) = \sigma_{\phi(u)}(\psi(u))$ and $\tau_{\phi}(\psi) \in Map(U, V)$ is defined by $(\tau_{\phi}(\psi))(u) = \tau_{\phi(u)}(\psi(u))$, for all $u\in U$. Let $\ker(\sigma) = \{k\in K \mid \sigma_{k}(h) = h \; \text{for all} \; h\in H\}$ and $Fix(\sigma) = \{h\in H\mid \sigma_{k}(h) = h, \; \text{for all} \; k\in K\}$. Similarly, we define the sets $\ker(\tau)$ and $Fix(\tau)$. Now, we find the center of the group $G$. 
	 
	 \begin{theorem}\label{zpcn}
	 	Let $G$ be the Zappa-Sz\'{e}p product of two groups $H$ and $K$. Then $Z(G) = \{(x,y)\in H\times K \mid x\in Fix(\sigma), y\in Fix(\tau), \sigma_{y} = i_{x^{-1}}, \tau_{x} = i_{y}\}$.
	 \end{theorem}
	 \begin{proof}
	 	Let $\Gamma = \{(x,y)\in H\times K \mid x\in Fix(\sigma), y\in Fix(\tau), \sigma_{y} = i_{x^{-1}}, \tau_{x} = i_{y}\}$. Let $(x,y)\in Z(G)$. Then for all $h\in H$, we have $(x,y)(h,1) = (h,1)(x,y)$. Therefore, $(x\sigma_{y}(h), \tau_{h}(y)) = (hx, y)$. Thus, we have
	 	\begin{equation}\label{s1e1}
	 	x\sigma_{y}(h) = hx \;\text{and}\; y\in Fix(\tau).
	 	\end{equation}
	 	Using the similar argument, for $k\in K$, we get 
	 	\begin{equation}\label{s1e2}
	 	\tau_{x}(k)y = yk \;\text{and}\; x\in Fix(\sigma).
	 	\end{equation} 
	 	Using the Equation \ref{s1e1}, we get $\sigma_{y}(h) =x^{-1}hx = i_{x^{-1}}(h)$, for all $h\in H$. Thus $\sigma_{y} = i_{x^{-1}}$. Using the similar argument and the Equation \ref{s1e2}, we get $\tau_{x} = i_{y}$. Hence, $(x,y)\in \Gamma$.
	 	
	 	Conversely, let $(x,y)\in \Gamma$ and $(h,k)\in G$. Then $(x,y)(h,k) = (x\sigma_{y}(h), \tau_{h}(y)k) = (x i_{x^{-1}}(h), yk)= (hx, (yky^{-1})y) = (h\sigma_{k}(x), i_{y}(k)y) = (h\sigma_{k}(x), \tau_{x}(k)y) = (h,k)(x,y)$. Thus $(x,y)\in Z(G)$. Hence, $Z(G) = \Gamma$.
	 \end{proof}
	  
	  	Let us define the sets $H^{\ast} = Fix(\sigma)\cap \ker(\tau)\cap Z(H)$ and $K^{\ast} = Fix(\tau)\cap \ker(\sigma)\cap Z(K)$. 
	\begin{corollary}\label{s1p1}
		Let $G$ be the Zappa-Sz\'{e}p product of two abelian groups $H$ and $K$. Then $Z(G) = H^{\ast} \times K^{\ast}$.
	\end{corollary}
	\begin{proof}
		Let $(x,y)\in Z(G)$. Then using the Theorem \ref{zpcn}, $x\in Fix(\sigma), y\in Fix(\tau), \sigma_{y} = i_{x^{-1}}, \tau_{x} = i_{y}$, for all $x\in H$ and $y\in K$. Since $H$ and $K$ are abelian groups, $i_{x} = I_{H}$ and $i_{y} = I_{K}$, for all $x\in H$ and $y\in K$. Therefore,  $\sigma_{y}(h) = h$ and $\tau_{x}(k) = k$, for all $h\in H$ and $k\in K$. Thus $y\in \ker(\sigma)$ and $x\in \ker(\tau)$. Hence, $x\in Fix(\sigma)\cap \ker(\tau)$ and $y\in Fix(\tau) \cap \ker(\sigma)$ and the result holds.
			\end{proof}

\begin{corollary}
	Let $G = H\rtimes_{\phi} K$ be the semidirect product of groups $H$ and $K$ with the group homomorphism $\phi: K\longrightarrow Aut(H)$. Then $Z(G) = \{(x,y)\in H\times K \mid x\in Fix(\sigma), y\in Z(K), \phi_{y} = i_{x^{-1}}\}$.
\end{corollary}
\begin{proof}
	 By taking $\sigma = \phi$ and $\tau$ the trivial action of $H$ on $K$ in the Corollary \ref{s1p1}, the proof follows immediately. 
\end{proof}
	
	\n Let $\mathcal{A}_{c}$ be the set of all matrices of the form $\begin{pmatrix}
	\alpha & \beta\\
	\gamma & \delta
	\end{pmatrix}$, where $\alpha \in Map(H,H)$, $\beta \in Hom(K,H), \gamma \in Hom(H,K)$ and $\delta \in Map(K, K)$ satisfy the following conditions,
	\begin{itemize}
		\item[$(A1)$] $\alpha(hh^{\prime}) = \alpha(h)\sigma_{\gamma(h)}(\alpha(h^{\prime}))$,
		\item[$(A2)$] $h^{-1}\alpha(h)\in Fix(\sigma)$ and $\gamma(h)\in Fix(\tau)$,
		\item[$(A3)$] $\sigma_{\gamma(h)} = i_{h^{-1}\alpha(h)}$ and $\tau_{h^{-1}\alpha(h)} = i_{\gamma(h)}$,
		\item[$(A4)$] $\beta(k)\in Fix(\sigma)$ and $\tau_{\beta(k)}(k^{-1})\delta(k)\in Fix(\tau)$,
		\item[$(A5)$] $\sigma_{\tau_{\beta(k)}(k^{-1})\delta(k)} = i_{(\beta(k))^{-1}}$ and $\tau_{\beta(k)} = i_{\tau_{\beta(k)}(k^{-1})\delta(k)}$,
		\item[$(A6)$] $h^{-1}\alpha(h)\beta(k) = \beta(k)h^{-1}\alpha(h)$ and $\gamma(h)\tau_{\beta(k)}(k^{-1})\delta(k) = \tau_{\beta(k)}(k^{-1})\delta(k)\gamma(h)$,
		\item[$(A7)$] $\delta(kk^{\prime}) = \tau_{\beta(k^{\prime})}(\delta(k))\delta(k^{\prime})$,
		\item[$(A8)$] $\beta(k)\sigma_{\delta(k)}(\alpha(h)) = \alpha(\sigma_{k}(h))\beta(\tau_{h}(k))$,
		\item[$(A9)$]  $\tau_{\alpha(h)}(\delta(k))\gamma(h) = \gamma(\sigma_{k}(h))\delta(\tau_{h}(k)$,
		\item[$(A10)$] for any $h^{\prime}k^{\prime}\in G$, there exists a unique $h\in H$ and $k\in K$ such that $h^{\prime} = \alpha(h)\beta(k)$ and $k^{\prime} = \gamma(h)\delta(k)$.
	\end{itemize}
	\n  for all $h, h^{\prime}\in H$ and $k,k^{\prime}\in K$. Then, the set $\mathcal{A}_{c}$ forms a group with the binary operation defined below,
	\begin{equation*}
	\begin{pmatrix}
	\alpha^{\prime} & \beta^{\prime}\\
	\gamma^{\prime} & \delta^{\prime}
	\end{pmatrix} \begin{pmatrix}
	\alpha & \beta\\
	\gamma & \delta
	\end{pmatrix} = \begin{pmatrix}
	\alpha^{\prime}\alpha + \beta^{\prime}\gamma & \alpha^{\prime}\beta +  \beta^{\prime}\delta\\
	\gamma^{\prime}\alpha + \delta^{\prime}\gamma & \gamma^{\prime}\beta + \delta^{\prime}\delta
	\end{pmatrix}.
	\end{equation*}

\begin{remark}\label{rem}
	Let  $\begin{pmatrix}
	\alpha & \beta\\
	\gamma & \delta
	\end{pmatrix}\in \mathcal{A}_{c}$ and $h\in \sigma(h)$ and $k\in Fix(\tau)$. Then using $(C4), (C5)$ and $(A2), (A4)$, we get $\alpha(h)\in Fix(\sigma)$, $\delta(k)\in Fix(\tau)$, $\alpha(h)\beta(k^{\prime})=\beta(k^{\prime})\alpha(h)$ and $\gamma(h^{\prime})\delta(k) = \delta(k)\gamma(h^{\prime})$, for all $h^{\prime}\in H$ and $k^{\prime}\in K$.
\end{remark}

	\begin{theorem}	
		Let $G$ be the Zappa-Sz\'{e}p product of two groups $H$ and $K$. Let $\mathcal{A}_{c}$ be as above. Then there is an isomorphism of groups between $Aut_{c}(G)$ and $\mathcal{A}_{c}$ given by $\theta \longleftrightarrow \begin{pmatrix}
		\alpha & \beta\\
		\gamma & \delta
		\end{pmatrix}$, where $\theta(h) =  \alpha(h)\gamma(h)$ and $\theta(k) = \beta(k)\delta(k)$, for all $h\in H$ and $k\in K$.
	\end{theorem}
\begin{proof}
	Let $\theta \in Aut_{c}(G)$. Then, define the maps $\alpha, \beta, \gamma$ and $\delta$ by means of $\theta(h) = \alpha(h)\gamma(h)$ and $\theta(k) = \beta(k)\delta(k)$, for all $h\in H$ and $k\in K$. Now, $h^{-1}\alpha(h)\gamma(h) = h^{-1}\theta(h)\in Z(G)$. Therefore, by Theorem \ref{zpcn}, $(A2)$ and $(A3)$ holds. Similarly, $\sigma_{k^{-1}}(\beta(k))\tau_{\beta(k)}(k^{-1})\delta(k) = k^{-1}\beta(k)\delta(k) = k^{-1}\theta(k)\in Z(G)$. Therefore, by Theorem \ref{zpcn}, $\sigma_{k^{-1}}(\beta(k))\in Fix(\sigma)$ and $\tau_{\beta(k)}(k^{-1})\delta(k) \in Fix(\tau)$. Now, $\sigma_{k^{-1}}(\beta(k)) = \sigma_{k}(\sigma_{k^{-1}}(\beta(k)))= \sigma_{kk^{-1}}(\beta(k)) = \sigma_{1}(\beta(k)) = \beta(k)$. Thus $(A4)$ and $(A5)$ holds.

	Using, $h^{-1}\alpha(h)\gamma(h) \in Z(G)$, we get $\beta(k)h^{-1}\alpha(h)\gamma(h) = h^{-1}\alpha(h)\gamma(h)\beta(k) = h^{-1}\alpha(h)$ $\sigma_{\gamma(h)}(\beta(k))\tau_{\beta(k)}(\gamma(h)) = h^{-1}\alpha(h)\beta(k)\gamma(h)$. This implies that $h^{-1}\alpha(h)\beta(k) = \beta(k)$ $h^{-1}\alpha(h)$. Using the similar argument for $\beta(k)\tau_{\beta(k)}(k^{-1})\delta(k)\in Z(G)$, we get $\gamma(h)\tau_{\beta(k)}(k^{-1})\delta(k) = \tau_{\beta(k)}(k^{-1})\delta(k)\gamma(h)$. Therefore, $(A6)$ holds. Note that, for all $h,h^{\prime}\in H$,
	\begin{align*}
	\alpha(hh^{\prime})\gamma(hh^{\prime}) =&\; \theta(hh^{\prime})\\
	 =&\; \theta(h)\theta(h^{\prime})\\
	  =&\; (\alpha(h)\gamma(h))(\alpha(h^{\prime})\gamma(h^{\prime}))\\
	   =&\; \alpha(h)\sigma_{\gamma(h)}(\alpha(h^{\prime})) \tau_{\alpha(h^{\prime})}(\gamma(h))\gamma(h^{\prime})\\
	    =&\; \alpha(h)\sigma_{\gamma(h)}(\alpha(h^{\prime}))\gamma(h)\gamma(h^{\prime}).
	\end{align*}
	  Using the uniqueness of representation, $\alpha(hh^{\prime}) = \alpha(h)\sigma_{\gamma(h)}(\alpha(h^{\prime}))$ and $\gamma(hh^{\prime}) = \gamma(h)\gamma(h^{\prime})$. Thus $\gamma \in Hom(H, K)$ and $(A1)$ holds. Using the similar argument for $\theta(kk^{\prime}) = \theta(k)\theta(k^{\prime})$, we get $\beta\in Hom(K, H)$ and $(A7)$ holds. Now,
	  \begin{align*}
	 \beta(k)\sigma_{\delta(k)}(\alpha(h))\tau_{\alpha(h)}(\delta(k))\gamma(h)=&\; \beta(k)\delta(k) \alpha(h)\gamma(h)\\ =&\;  \theta(k)\theta(h)\\
	  =&\; \theta(kh)\\ =&\; \theta(\sigma_{k}(h)\tau_{h}(k))\\
	   =&\; \theta(\sigma_{k}(h))\theta(\tau_{h}(k))\\
	    =&\; \alpha(\sigma_{k}(h))\gamma(\sigma_{k}(h)) \beta(\tau_{h}(k))\delta(\tau_{h}(k))\\
	     =&\; \alpha(\sigma_{k}(h))\beta(\tau_{h}(k))\gamma(\sigma_{k}(h))\delta(\tau_{h}(k)).
	  \end{align*}
	   Using the uniqueness of representation, we get $(A8)$ and $(A9)$. Since $\theta$ is onto, $(A10)$ holds. Thus, by the uniqueness of representation, to every $\theta\in Aut_{c}(G)$, we can associate a unique matrix $\begin{pmatrix}
	\alpha & \beta\\
	\gamma & \delta	
	\end{pmatrix} \in \mathcal{A}_{c}$. Thus we have a map $\eta: Aut_{c}(G) \longrightarrow \mathcal{A}_{c}$ given by $\theta \mapsto \begin{pmatrix}
	\alpha & \beta \\ \gamma & \delta
	\end{pmatrix}$.
	
	Now, let $\begin{pmatrix}
	\alpha & \beta\\
	\gamma & \delta	
	\end{pmatrix} \in \mathcal{A}_{c}$ satisfying the conditions $(A1)-(A10)$. Then, define $\theta : G \longrightarrow G$ by $\theta(h) = \alpha(h)\gamma(h)$ and $\theta(k) = \beta(k)\delta(k)$, for all $h\in H$ and $k\in K$. Using  $(A1)$, $(A7)-(A9)$, one can easily check that $\theta$ is an endomorphism of $G$. Also, by $(A10)$, $\theta$ is onto. Now, let $\theta(hk) = 1$. Then $1 = \theta(h)\theta(k) = (\alpha(h)\beta(k))(\gamma(h)\delta(k))$. Thus, by the uniqueness of representation, we get $\alpha(h)\beta(k) = 1$ and $\gamma(h)\delta(k) = 1$. By \cite[Proposition 2.1, p. 4]{rv}, we get $\alpha(h) = 1 = \gamma(h)$ and $\beta(k) = 1 = \delta(k)$. Therefore, by \cite[Lemma 2.2, p. 4]{rv}, we get $h = 1 =k$ and so, $\theta$ is one-one. Thus $\theta \in Aut(G)$. Now, let $g = hk\in G$. Then $g^{-1}\theta(g) = k^{-1}h^{-1}\alpha(h)\beta(k)\gamma(h)\delta(k) = \sigma_{k^{-1}}(h^{-1}\alpha(h)\beta(k)) \tau_{h^{-1}\alpha(h)\beta(k)}(k^{-1}) \gamma(h)\delta(k)$. Using $(A2)-(A6)$, we get $g^{-1}\theta(g)\in Z(G)$. Hence, $\theta\in Aut_{c}(G)$. Therefore, $\eta$ is a bijection. Now, for all $g = hk\in G$, we have
	\begin{align*}
	\theta^{\prime}\theta(hk) =&\; \theta^{\prime}(\alpha(h)\beta(k)\gamma(h)\delta(k))\\
	=&\; \alpha^{\prime}(\alpha(h)\beta(k))\beta^{\prime}(\gamma(h)\delta(k))\gamma^{\prime}(\alpha(h)\beta(k))\delta^{\prime}(\gamma(h)\delta(k))\\
	=&\; \alpha^{\prime}(\alpha(h))\sigma_{\gamma^{\prime}(\alpha(h))}(\alpha^{\prime}(\beta(k)))\beta^{\prime}(\gamma(h)\delta(k))\gamma^{\prime}(\alpha(h)\beta(k))\tau_{\beta^{\prime}(\delta(k))}(\delta^{\prime}(\gamma(h)))\delta^{\prime}(\delta(k))\\ 
	=&\; \alpha^{\prime}(\alpha(h))\alpha^{\prime}(\beta(k))\beta^{\prime}(\gamma(h))\beta^{\prime}(\delta(k))\gamma^{\prime}(\alpha(h))\gamma^{\prime}(\beta(k))\delta^{\prime}(\gamma(h))\delta^{\prime}(\delta(k))\\
	&\; (\text{Using $(A2)$, $(A4)$ and the Remark \ref{rem}})\\
	=&\; \alpha^{\prime}\alpha(h)\beta^{\prime}\gamma(h)\alpha^{\prime}\beta(k)\beta^{\prime}\delta(k)\gamma^{\prime}\alpha(h)\delta^{\prime}\gamma(h)\gamma^{\prime}\beta(k)\delta^{\prime}(\delta(k))\; (\text{Using the Remark \ref{rem}})\\
	=&\; (\alpha^{\prime}\alpha + \beta^{\prime}\gamma)(h) (\alpha^{\prime}\beta + \beta^{\prime}\delta)(k) (\gamma^{\prime}\alpha + \delta^{\prime}\gamma)(h) (\gamma^{\prime}\beta + \delta^{\prime}\delta)(k).
		\end{align*}  
	Writing $hk$ as $\begin{pmatrix}
	h\\ k
	\end{pmatrix}$, and the map $\begin{pmatrix}
	\alpha & \beta \\ \gamma & \delta
	\end{pmatrix}$, we get 
	\[\theta(hk) = \begin{pmatrix}
	\alpha & \beta \\ \gamma & \delta
	\end{pmatrix}\begin{pmatrix}
	h \\ k
	\end{pmatrix} = \begin{pmatrix}
	\alpha(h)\beta(k)\\ \gamma(h)\delta(k)
	\end{pmatrix}\]
	and \[\theta^{\prime}\theta(hk) = \begin{pmatrix}
	\alpha^{\prime} & \beta^{\prime}\\ \gamma^{\prime} & \delta^{\prime}
	\end{pmatrix}\begin{pmatrix}
	\alpha(h)\beta(k)\\ \gamma(h)\delta(k)
	\end{pmatrix} = \begin{pmatrix}
	\alpha^{\prime}\alpha + \beta^{\prime}\gamma & \alpha^{\prime}\beta + \beta^{\prime}\delta\\ \gamma^{\prime}\alpha + \delta^{\prime}\gamma & \gamma^{\prime}\beta + \delta^{\prime}\delta
	\end{pmatrix}\begin{pmatrix}
	h \\ k
	\end{pmatrix}. \]
	Therefore, $\eta(\theta^{\prime}\theta) = \begin{pmatrix}
	\alpha^{\prime}\alpha + \beta^{\prime}\gamma & \alpha^{\prime}\beta + \beta^{\prime}\delta\\ \gamma^{\prime}\alpha + \delta^{\prime}\gamma & \gamma^{\prime}\beta + \delta^{\prime}\delta
	\end{pmatrix} = \eta(\theta^{\prime})\eta(\theta)$. Hence, $\eta$ is an isomorphism of groups.
\end{proof}

	\n We will identify the central automorphisms of $G$ with the corresponding matrices in $\mathcal{A}_{c}$. Note that, if $h^{-1}\alpha(h)\in H^{\ast}$, then $\tau_{\alpha(h)}(k) = \tau_{h}(k)$, for all $h\in H$ and $k\in K$. Also, if $k^{-1}\delta(k)\in K^{\ast}$, then $\sigma_{\delta(k)}(h) = \sigma_{k}(h)$, for all $h\in H$ and $k\in K$. Let 
	\begin{align*}
	P =&\; \{\alpha\in Aut_{c}(H) \mid \sigma_{k}(\alpha(h)) = \alpha(\sigma_{k}(h)), h^{-1}\alpha(h)\in H^{\ast}\; \forall\; h\in H, k\in K\},\\
	Q =&\; \{\beta\in Hom(K,H^{\ast}) \mid \beta(k) = \beta(\tau_{h}(k))\; \forall\; h\in H, k\in K\},\\
	R =&\; \{\gamma \in Hom(H,K^{\ast}) \mid \gamma(\sigma_{k}(h)) = \gamma(h)\; \forall\; h\in H, k\in K\},\\
	S =&\; \{\delta\in Aut_{c}(K) \mid \tau_{h}(\delta(k)) = \delta(\tau_{h}(k)), k^{-1}\delta(k)\in K^{\ast}\; \forall\; h\in H, k\in K\}.
\end{align*}
	Then one can easily check that $P$, $Q$, $R$ and $S$ are all subgroups of the group $Aut_{c}(G)$. Let
	\begin{center}
		$\begin{matrix}
		A = \left\{\begin{pmatrix}
		\alpha & 0\\
		0 & 1
		\end{pmatrix}\mid \alpha\in P\right\}, & B = \left\{\begin{pmatrix}
		1 & \beta\\
		0 & 1
		\end{pmatrix}\mid \beta\in Q\right\},\\
		C = \left\{\begin{pmatrix}
		1 & 0 \\
		\gamma & 1
		\end{pmatrix}\mid \gamma\in R\right\}, & D = \left\{\begin{pmatrix}
		1 & 0\\
		0 & \delta
		\end{pmatrix}\mid \delta\in S\right\}.

		\end{matrix}$
	\end{center}
		\n be the corresponding subsets of $\mathcal{A}_{c}$, where $0$ is the trivial group homomorphism and $1$ is the identity automorphism of groups. Then one can easily check that $A$, $B$, $C$ and $D$ are subgroups of $\mathcal{A}_{c}$. Note that $A$ and $D$ normalize $B$ and $C$. By the similar argument as in the proof of the Theorem \cite[Theorem 2.4, p. 7]{rv}, we have the following Theorem.
		\begin{theorem}\label{s1t1}
			Let $G$ be the Zappa-Sz\'{e}p product of two groups $H$ and $K$. Let $A, B, C$ and $D$ be defined as above. Then, if $1-\beta\gamma \in P$, for all maps $\beta$ and $\gamma$, then $ABCD = \mathcal{A}_{c}$.
		\end{theorem}
	\begin{corollary}
		Let $G$ be the Zappa-Sz\'{e}p product of two abelian groups $H$ and $K$. Let $A, B, C$ and $D$ be defined as above. Then, if $1-\beta\gamma \in P$, for all maps $\beta$ and $\gamma$, then $ABCD = \mathcal{A}_{c}$, where  
		\[\mathcal{A}_{c} = \left\{\begin{pmatrix}
		\alpha & \beta \\ \gamma & \delta
		\end{pmatrix} \mid \begin{matrix}
		\alpha \in Epi(H) & \beta\in Hom(K,H^{\ast})\\ \gamma \in Hom(H, K^{\ast} & \delta\in Epi(K))
		\end{matrix} \right\}\]
		and the maps $\alpha, \beta, \gamma$ and $\delta$ satisfies $h^{-1}\alpha(h)\in H^{\ast}$,  $k^{-1}\delta(k)\in K^{\ast}$ and $(A8)- (A10)$.
	\end{corollary}
	
\section*{Computation}
	Consider a group $G = \langle a,b,c,d,e \mid a^{p} = b^{p} = c^{p} = d^{p} = e^{p} = 1, ab= ba, ac=ca, ad=da, ae= eac, bc= cb, bd =db, be= ebd, cd= dc, ce=ec, de =ed\rangle$, where $p$ is an odd prime. Note that $|G| = p^{5}$. Let $H = \langle a,b,d \mid a^{p} = b^{p}  = d^{p} = 1, ab= ba, ad=da, bd =db \rangle$, $K = \langle c,e \mid c^{p} = e^{p} = 1, ec=ce\rangle$. The mutual actions of $H$ and $K$ are defined for the generators of $H$ and $K$ by 
	\begin{equation}\label{s4e2}
	\sigma_{j}(i) = \left\{\begin{array}{ll}
	bd^{-1}, & \text{if}\; i=b, j = e\\
	i, & otherwise
	\end{array}\right. \; \text{and}\; \tau_{i}(j) = \left\{\begin{array}{ll}
	ec^{-1}, & \text{if}\; i=a, j = e\\
	j, & otherwise.
	\end{array}\right.
	\end{equation}	
	Then $G$ is a Zappa-Sz\'{e}p product of groups $H$ and $K$. Clearly, $Fix(\sigma) = \langle a,d\rangle$, $\ker(\sigma) = \langle c\rangle$, $Fix(\tau) = \langle c\rangle$ and $\ker(\tau) = \langle b,d\rangle$. Therefore, $H^{\ast} = \langle d\rangle$, $K^{\ast} = \langle c\rangle$ and so $Z(G) = \langle c,d\rangle$.

	Now, let $\alpha\in P$. Then $h^{-1}\alpha(h)\in H^{\ast} = \langle d\rangle$, for all $h\in H$. Suppose that $\alpha(a) = ad^{\nu_{1}}$, $\alpha(b) = bd^{\nu_{2}}$ and $\alpha(d) = d^{\nu_{3}}$, where $0\le \nu_{1}, \nu_{2}\le p-1$ and $1\le \nu_{3}\le p-1$. Then the map $\alpha$ can be identified with an element of $GL(3,p)$. Also, using the relations in (\ref{s4e2}), $\sigma_{e}(\alpha(a)) = \alpha(\sigma_{e}(a))$, $\sigma_{e}(\alpha(d)) = \alpha(\sigma_{e}(d))$ and $\sigma_{c}(\alpha(h)) = \alpha(\sigma_{c}(h))$ holds  trivially for all $h\in H$. Now, using $\sigma_{e}(\alpha(b)) = \alpha(\sigma_{e}(b))$, we get $bd^{\nu_{2}-1} = \sigma_{e}(\alpha(b)) = \alpha(\sigma_{e}(b)) = \alpha(bd^{-1}) = bd^{\nu_{2}}d^{-\nu_{3}} = bd^{\nu_{2}-\nu_{3}}$. Thus $\nu_{2}-1\equiv \nu_{2}-\nu_{3} \Mod{p}$ which implies that $\nu_{3} \equiv 1\Mod{p}$. Therefore, $\alpha$ can be identified with the matrix $\begin{pmatrix}
	1& 0 & 0\\
	0 & 1& 0\\
	\nu_{1}& \nu_{2} & 1
	\end{pmatrix}$, where $0\le u,v\le p-1$. Let $x = \begin{pmatrix}
	1& 0 & 0\\
	0 & 1& 0\\
	0& \nu_{2} & 1
	\end{pmatrix}$  and $y = \begin{pmatrix}
	1& 0 & 0\\
	0 & 1& 0\\
	\nu_{1}& 0 & 1
	\end{pmatrix}$. Then $A\simeq \langle x,y \mid x^{p} = y^{p} = 1, xy=yx\rangle \simeq \mathbb{Z}_{p}\times \mathbb{Z}_{p}$.
	
	Let $\beta\in Q$. Then $Im(\beta)\le H^{\ast} =\langle d\rangle$. Let $\beta$ be defined as $\beta(c) = d^{\lambda_{1}}$ and $\beta(e) = d^{\lambda_{2}}$, where $0\le \lambda_{1}, \lambda_{2}\le p-1$. Since $K\simeq \mathbb{Z}_{p}\times \mathbb{Z}_{p}$ and $H^{\ast}\simeq \mathbb{Z}_{p}$, the map $\beta$ can be identified with an element $\begin{pmatrix}
	\lambda_{1} & \lambda_{2}
	\end{pmatrix}$ of $M_{1\times 2}(\mathbb{Z}_{p})$. Clearly, using the relations in (\ref{s4e2}), $\beta(\tau_{a}(c)) = \beta(c)$ and $\beta(\tau_{h}(k)) = \beta(k)$ holds trivially for all $h\in \{b,d\}$ and $k\in K$. Now, using $\beta(\tau_{a}(e)) = \beta(e)$, we get $\beta(e) = \beta(\tau_{a}(e)) = \beta(ec^{-1}) = \beta(e)\beta(c^{-1})$. Thus, $\beta(c) = 1$ which implies that $\lambda_{1} = 0$. Hence, the map $\beta$ can be identified with the matrix $\begin{pmatrix}
	0 & \lambda_{2}
	\end{pmatrix}\in M_{1\times 2}(\mathbb{Z}_{p})$. Let $u = \begin{pmatrix}
	0 & \lambda_{2}
	\end{pmatrix}$. Then, $B \simeq \langle  u \mid u^{p} = 1\rangle \simeq \mathbb{Z}_{p}$.
	
	Now, let $\gamma\in R$. Then $Im(\gamma)\in K^{\ast} = \langle c \rangle$. So, we define $\gamma(a) = c^{\mu_{1}}$,  $\gamma(b) = c^{\mu_{2}}$ and $\gamma(d) = c^{\mu_{3}}$, where $0\le \mu_{1}, \mu_{2}, \mu_{3}\le p-1$. Since $H\simeq \mathbb{Z}_{p}\times \mathbb{Z}_{p} \times \mathbb{Z}_{p}$ and $K^{\ast}\simeq \mathbb{Z}_{p}$, the map $\gamma$ can be identified with an element $\begin{pmatrix}
	\mu_{1} & \mu_{2} & \mu_{3}
	\end{pmatrix}$ of $M_{1\times 3}(\mathbb{Z}_{p})$.  Clearly, using the relations in (\ref{s4e2}), $\gamma(\sigma_{e}(a)) = \gamma(a)$, $\gamma(\sigma_{e}(d)) = \gamma(d)$, $\gamma(\sigma_{c}(h)) = \gamma(h)$ holds trivially for all $h\in H$. Now, using $\gamma(b) = \gamma(\sigma_{e}(b)) = \gamma(bd^{-1}) = \gamma(b)\gamma(d^{-1})$, we get $\gamma(d) = 1$. Therefore, $\mu_{3} = 0$ and $\gamma$ can be identified with the matrix $\begin{pmatrix}
	\mu_{1} & \mu_{2} & 0
	\end{pmatrix}\in M_{1\times 3}(\mathbb{Z}_{p})$. Let $z = \begin{pmatrix}
	\mu_{1} & 0 & 0
	\end{pmatrix}$ and $w = \begin{pmatrix}
	0 & \mu_{2} & 0
	\end{pmatrix}$. Then, $R \simeq \langle z,w \mid z^{p} = w^{p} = 1, zw=wz\rangle  \simeq \mathbb{Z}_{p}\times \mathbb{Z}_{p}$.
	
	Let $\delta\in S$. Then $k^{-1}\delta(k)\in K^{\ast} = \langle c\rangle$, for all $k\in K$. Suppose that $\delta(c) = c^{\rho_{1}}$ and $\delta(e) = ec^{\rho_{2}}$, where $1\le \rho_{1}\le p-1$ and $0\le \rho_{2}\le p-1$. Then the map $\delta$ can be identified with an element of $GL(2,p)$. Clearly, using the relations in (\ref{s4e2}), $\tau_{b}(\delta(e)) = \delta(\tau_{b}(e))$, $\tau_{d}(\delta(e)) = \delta(\tau_{d}(e))$ and $\tau_{h}(\delta(c)) = \delta(\tau_{h}(c))$ holds trivially for all $h\in H$. Now, using $\tau_{a}(\delta(e)) = \delta(\tau_{a}(e))$, we get $ec^{-1+\rho_{2}} = \tau_{a}(\delta(e)) = \delta(\tau_{a}(e)) = \delta(ec^{-1}) = ec^{\rho_{2}-\rho_{1}}$. Thus $\rho_{1} \equiv 1\Mod{p}$. Therefore, $\delta(c) = c$, $\delta(e) = ec^{\rho_{2}}$, where $0\le \rho_{2}\le p-1$. Hence, the map $\delta$ can be identified with the matrix $\begin{pmatrix}
	1 & \rho_{2}\\ 0&1
	\end{pmatrix} \in GL(2,p)$. Let $t = \begin{pmatrix}
	1 & \rho_{2}\\ 0&1
	\end{pmatrix}$. Then $D \simeq \langle t \mid t^{p} = 1\rangle \simeq \mathbb{Z}_{p}$.
	
	Also, one can easily verify that $\alpha\beta = \beta = \beta\delta$, $\gamma\alpha = \gamma = \delta\gamma$, and $\beta\gamma = 0 = \gamma\beta$. Therefore, $Aut_{c}(G) \simeq A\times B \times C\times D$. Hence, $Aut_{c}(G)\simeq \mathbb{Z}_{p}\times \mathbb{Z}_{p} \times \mathbb{Z}_{p}\times \mathbb{Z}_{p}\times \mathbb{Z}_{p}\times \mathbb{Z}_{p}$.

\vspace{0.2 cm}

\noindent \textbf{Acknowledgment:}{ The first author is supported by the Senior Research Fellowship of UGC, India.}


\begin{thebibliography}{9}
\bibitem{sz3} J. Sz\'{e}p, On the structure of groups which can be represented as the product of two subgroups, \textit{Acta Sci. Math. (Szeged)}, \textbf{12} (1950), 57--61.
\end{thebibliography}
\end{document}